\documentclass[11pt, reqno]{amsart}
\usepackage[utf8]{inputenc}
\usepackage{indentfirst, amssymb, amsmath, amsthm, mathrsfs, setspace, indentfirst, enumerate,  mathrsfs, amsmath, amsthm}
\usepackage[colorlinks=true,linkcolor=purple, citecolor=blue,urlcolor=magenta]{hyperref}
\usepackage{graphicx}  
\usepackage{cite}    
\usepackage{float}         
\usepackage{xcolor}        
\usepackage{colortbl}     
\usepackage{multirow}     
\usepackage{tikz}          

\definecolor{sectionlink}{RGB}{0,100,200} 
\newtheorem{conj}{Conjecture}[section]
\newtheorem{prop}{Proposition}[section]
\newtheorem{theo}{Theorem}[section]
\newtheorem{lem}{Lemma}[section]
\newtheorem{cor}{Corollary}[section]

\newtheorem{exm}{Example}[section]
\newtheorem{defi}{Definition}[section]
\newtheorem{rem}{Remark}[section]
\newtheorem*{theoA}{Theorem A}

\numberwithin{equation}{section}
\newcommand{\beas}{\begin{eqnarray*}}
\newcommand{\eeas}{\end{eqnarray*}}
\newcommand{\bea}{\begin{eqnarray}}
\newcommand{\eea}{\end{eqnarray}}

\begin{document}

\title[Sharp Bounds of $H_{2,2}(F)$ for $\beta$-Spirallike Mappings on Banach Spaces]{Second Hankel Determinant for $\beta$-Spirallike Convex Mappings in Complex Banach Spaces}
\author[ M.B. Ahamed, N. Sarkar and P. Das]{Molla Basir Ahamed, Nabadwip Sarkar and Pradip Das}

\address{Molla Basir Ahamed, Department of Mathematics, Jadavpur  University, Kolkata-700032, West Bengal, India}
\email{mbahamed.math@jadavpuruniversity.in}

\address{Nabadwip Sarkar, Amity School of Applied Sciences, Amity University Mumbai, Panvel, Navi Mumbai, Maharashtra-410206, India}
\email{nsarkar@mum.amity.edu, nabadwipsarkar52@gmail.com}

\address{Pradip Das, Department of Mathematics, Raiganj University, Raiganj, West Bengal-733134, India.}
\email{pradipsmath@gmail.com}

\makeatletter
\@namedef{subjclassname@2020}{\textup{2020} Mathematics Subject Classification}
\makeatother


\begin{abstract}
	We establish the bound for the second-order Hankel determinant $H_{2,2}(F) = A_2 A_4 - A_3^2$ 
	associated with the class $\mathcal{C}_{B}^{\beta}(\mathbb{B})$ of normalized $\beta$-spirallike quasi-convex 
	mappings of type $B$ on the open unit ball $\mathbb{B}$ of a complex Banach space. By utilizing a generalized framework 
	based on a directional slice homogeneous polynomial expansion, we eliminate the standard, restrictive 
	assumption that the mapping is of the form $F(x) = g(x)x$. Under these weaker operational conditions, 
	we parameterize the targeted scalar invariants $A_n$ via the classical Carath\'{e}odory functional parameters. 
	A rigorous optimization analysis proves that the established upper bound is strictly sharp for the classical 
	non-spirallike case $\beta = 0$, yielding a maximal value of $1/8$. This sharp bound is verified by constructing 
	explicit multi-dimensional extremal mappings that lift the corresponding single-variable convex profile. 
	Finally, an unresolved open question regarding the exact variational behavior for $\beta \neq 0$ is formulated.
\end{abstract}

\maketitle

\section{\bf Introduction}It is well known that many classical results in the theory of a single complex variable do not naturally extend to higher dimensions. For instance, Cartan~\cite{Cartan-1933} pointed out the failure of the celebrated Bieberbach conjecture in several complex variables. This failure, along with other classical counterexamples, suggests that imposing additional geometric conditions, such as convexity or starlikeness, is essential when extending one-dimensional results to higher dimensions.\vspace{1.2mm}

{In \cite{Gong1999}, Gong posed the following conjecture.
	\begin{conj}\label{Conj-1.1} If $f : \mathbb{U}^n \to \mathbb{C}^n$ is a normalized biholomorphic starlike mapping, then
		\[\frac{\|D^m f(0)(z^m)\|}{m!} \le m \|z\|^m,\; z \in \mathbb{U}^n, \; m = 2, 3, \dots	\]
	\end{conj}
	In view of a result in \cite{Roper-TAMS-1999}, Conjecture~\ref{Conj-1.1} does not hold in general for normalized biholomorphic starlike mappings on $\mathbb{B}^n$ in $\mathbb{C}^n$ (see \cite{GHK2002}). So, in \cite{GHK2002}, Graham \textit{et al.} proposed the following conjecture.
	\begin{conj}\label{Conj-1.2}
		If $f : \mathbb{B}_n \to \mathbb{C}^n$ is a normalized biholomorphic starlike mapping, then
		\[
		\frac{|T_x(D^m f(0)(x^m))|}{m!} \le m \|x\|^m,\; x \in \mathbb{B}_n,\;T_x \in T(x),\; m = 2, 3, \dots
		\]
		where $\mathbb{B}_n$ is the unit ball of $\mathbb{C}^n$ with respect to an arbitrary norm.
	\end{conj}
	We remark that both Conjecture \ref{Conj-1.1} and Conjecture \ref{Conj-1.2} are called the Bieberbach conjectures in several complex variables. Some best-possible results concerning the coefficient estimates for subclasses of holomorphic mappings in several variables were obtained in the works of Graham et al. \cite{GHHK2014,GHK2002,GKK2003}, Hamada and Honda \cite{HH2008}, Hamada \textit{et al.} \cite{HHK2006}, Kohr \cite{Kohr1998}, Liu \textit{et al.} \cite{LLX2015} and Xu and Liu \cite{XL2009a}.\vspace{1.2mm}
	
 It is natural to ask why we transition from the classical single complex variable setting to a setting of complex Banach spaces, and how the concept of differentiation adapts to this infinite-dimensional setting. \vspace{1.2mm}
	
	In a single complex variable, the derivative of a function $f: \mathbb{C} \to \mathbb{C}$ at a point $z_0$ is simply a complex number $f'(z_0)$. This representation is highly specialized because the space of linear transformations from $\mathbb{C}$ to $\mathbb{C}$ is isomorphic to $\mathbb{C}$ itself. However, when studying physical systems, control theory, or functional equations, state spaces are frequently infinite-dimensional. To analyze holomorphic dynamics and geometric properties in these settings, we require the framework of a complex Banach space $X$ endowed with the norm $\|\cdot\|$. \vspace{1.2mm}
	
	Going to higher dimensions or infinite-dimensional spaces changes the fundamental nature of the derivative. The derivative is no longer a scalar; it is a continuous linear operator. For a holomorphic mapping $g: \mathbb{B} \to X$, where $\mathbb{B} = \{z \in X : \Vert{}z\Vert{} < 1\}$,  the Fréchet derivative $Dg(z)$ at a point $z \in B$ is a bounded linear operator from $X$ into $X$ (i.e., $Dg(z) \in \mathcal{L}(X, X)$). It represents the best local linear approximation of the mapping $g$ near $z$, satisfying
	\[
	\lim_{h \to 0} \frac{\|g(z+h) - g(z) - Dg(z)(h)\|}{\|h\|} = 0.
	\]
	For the finite-dimensional case, let $\mathbb{C}^n$ represent the space of $n$ complex variables, where elements are written as column vectors $z = (z_1, z_2, \dots, z_n)^T$. Let $\mathbb{U}^n$ denote the open unit polydisc in $\mathbb{C}^n$. We denote the boundary and the distinguished boundary of $\mathbb{U}^n$ by $\partial \mathbb{U}^n$ and $\partial_0 \mathbb{U}^n$, respectively.\vspace{1.2mm} 
	
	In particular, when $X = \mathbb{C}^n$, every linear operator can be identified with an $n \times n$ matrix with respect to the standard basis. Thus, the Fréchet derivative $Dg(z)$ can be concretely realized as the classical Jacobian 
	matrix
	\[
	Dg(z) = \left[ \frac{\partial g_j(z)}{\partial z_k} \right]_{1 \le j, k \le n}.
	\]
	Similarly, the $m$-th order Fréchet derivative $D^m g(z)$ is a bounded symmetric $m$-linear operator. When evaluated at the vector $a \in \mathbb{C}^n$ for its first $m-1$ arguments, the resulting object $D^m g(z)(a^{m-1}, \cdot)$ becomes a linear operator from $\mathbb{C}^n$ to $\mathbb{C}^n$. Consequently, it admits the matrix representation
	\[
	D^m g(z)(a^{m-1}, \cdot) = \left[ \sum_{p_1, \dots, p_{m-1}=1}^n \frac{\partial^m g_j(z)}{\partial z_k \partial z_{p_1} \dots \partial z_{p_{m-1}}} a_{p_1} \dots a_{p_{m-1}} \right]_{1 \le j, k \le n}.
	\]
	This formulation bridges the abstract coordinate-free calculus on complex Banach spaces with the classical coordinate-based matrix calculus in $\mathbb{C}^n$.}
	{\begin{exm} Let $z = (z_1, z_2, z_3)^T \in \mathbb{C}^3$ and $a = (a_1, a_2, a_3)^T \in \mathbb{C}^3$.We define the quadratic mapping $g(z) = (g_1(z), g_2(z), g_3(z))^T$ by 
			\begin{align*}
				g(z) = \begin{pmatrix} z_1^2 + z_2 z_3 \\ z_2^2 \\ z_3^2 + z_1 z_2 \end{pmatrix}.
			\end{align*}
			For $m=2$, the term inside the matrix elements is a single sum over $p_1$ from $1$ to $3$, 
			\begin{align*}
				\sum_{p_1=1}^3 \frac{\partial^2 g_j(z)}{\partial z_k \partial z_{p_1}} a_{p_1}.
			\end{align*}
			The second Fréchet derivative operator $D^2 g(z)(a, \cdot)$ is represented by the following matrix
			\[
			D^2 g(z)(a, \cdot) = 
			\begin{pmatrix}
				2a_1 & a_3 & a_2 \\[1.5ex]
				0 & 2a_2 & 0 \\[1.5ex]
				a_2 & a_1 & 2a_3
			\end{pmatrix}.
			\]
	\end{exm}
	We denote by $\mathcal{L}(X,Y)$ the Banach space of all bounded linear operators from $X$ into another complex Banach space $Y$. The identity operator on $X$ is denoted by $I$.\vspace{1.2mm}
	
	For each $x \in X \setminus \{0\}$, we define the set
	\[
	T_x = \{ l_x \in L(X, \mathbb{C}) : l_x(x) = \|x\|, \, \|l_x\| = 1 \},
	\]
	where $\mathcal{L}(X, Y)$ denotes the space of all continuous linear operators from a complex Banach space $X$ into a complex Banach space $Y$. Let $I$ denote the identity operator on $X$. By the Hahn-Banach theorem, the set $T_x$ is nonempty for every $x \in X \setminus \{0\}$.\vspace{1.2mm}
	
	Let $\mathcal{H}(\Omega, \Omega')$ denote the space of all holomorphic mappings from a domain $\Omega \subseteq X$ into a domain $\Omega' \subseteq Y$, and write $\mathcal{H}(\Omega) := \mathcal{H}(\Omega, X)$. For a mapping $g \in \mathcal{H}(\mathbb{B})$ and each $k \in \mathbb{N}$, there exists a bounded symmetric $k$-linear mapping 
	\[
	D^k g(z) : \prod_{j=1}^k X \to X,
	\]
	which is the $k$-th order Fréchet derivative of $g$ at $z$, such that $g$ admits the Taylor series expansion
	\[
	g(w) = \sum_{k=0}^{\infty} \frac{1}{k!} D^k g(z)\left((w - z)^k\right)
	\]
	for all $w$ in a neighborhood of $z$. Here, we define $D^0 g(z)\left((w - z)^0\right) = g(z)$, and for $k \ge 1$, 
	\[
	D^k g(z)\left((w - z)^k\right) = D^k g(z)(\underbrace{w - z, w - z, \dots, w - z}_{k \text{ times}}).
	\]
	
	On a bounded circular domain $\Omega \subset \mathbb{C}^n$, the first Fréchet derivative $Dg(z)$ and the $m$-th Fréchet derivative $D^m g(z)(a^{m-1}, \cdot)$ of a holomorphic mapping $g : \Omega \to \mathbb{C}^n$ are represented in matrix form as
	\[
	Dg(z) = \left[ \frac{\partial g_j(z)}{\partial z_k} \right]_{1 \le j, k \le n}
	\]
	and
	\[
	D^m g(z)(a^{m-1}, \cdot) = \left[ \sum_{p_1, \dots, p_{m-1}=1}^n \frac{\partial^m g_j(z)}{\partial z_k \partial z_{p_1} \dots \partial z_{p_{m-1}}} a_{p_1} \dots a_{p_{m-1}} \right]_{1 \le j, k \le n},
	\]
	respectively, where $g(z) = (g_1(z), \dots, g_n(z))^T$ and $a = (a_1, \dots, a_n)^T \in \mathbb{C}^n$.\vspace{1.2mm}
	
	A mapping $g \in \mathcal{H}(\mathbb{B})$ is said to be normalized if $g(0) = 0$ and $Dg(0) = I$. A holomorphic mapping $g : \Omega \to X$ is biholomorphic if its inverse $g^{-1}$ exists and is holomorphic on $g(\Omega)$. If $Dg(z)$ has a bounded inverse for each $z \in \Omega$, then $g$ is said to be locally biholomorphic. We denote by $\mathcal{S}(\mathbb{B})$ the class of all normalized biholomorphic mappings from the unit ball $\mathbb{B}$ into $X$.}
\begin{defi} A mapping $F\in\mathcal{H}(\mathbb{B})$ is called \emph{biholomorphic} if $F(\mathbb{B})$ is a domain in $X$ and the inverse mapping
\[
F^{-1}:F(\mathbb{B})\rightarrow \mathbb{B}
\]
exists and is holomorphic. Moreover, $F$ is said to be \emph{locally biholomorphic} whenever the Fr\'echet derivative $DF(x)$ is invertible with bounded inverse for every $x\in \mathbb{B}$.
\end{defi}

A holomorphic mapping $F:B\rightarrow X$ is said to be \emph{normalized} if
\[
F(0)=0\;\text{and}\;DF(0)=I.
\]
Several important subclasses of normalized biholomorphic mappings in Banach spaces can be found in the monograph~\cite{GK2003}.\vspace{1.2mm}

Let $x_0$ be an arbitrary vector on the unit sphere $\partial\mathbb{B}$ and let $T_{x_0}$ be an associated Hahn--Banach functional. Correspondingly, we define the baseline scalar invariant $A_1=1$, and introduce the higher-order Taylor coefficients for $n=2,3,4, \ldots$, via the directional evaluations
\begin{align}
	\label{eq:An}
	A_n=\frac{1}{n!}T_{x_0}\left(D^nF(0)(x_0,\ldots,x_0)
	\right),
\end{align}
where the vector $x_0$ appears exactly $n$ times in the multilinear form $D^nF(0)$.\vspace{1.2mm}

We shall also require the following definitions.

\begin{defi}\cite{G2003}
Let $F:\mathbb{B} \rightarrow  X$ be a normalized locally biholomorphic mapping. If $Re\;\{T_x((DF(x))^{-1}(D^2F(x)(x^2)+DF(x)x)\}\geq 0$, $x\in \mathbb{B}\setminus\{0\}$, $T_x\in T(x)$, then F is called quasi-convex mapping of type $B$ on $\mathbb{B}$.
\end{defi}
The following definition extends the class of quasi-convex mappings of type B by introducing a spiral parameter.

\begin{defi}
Let \(X\) be a complex Banach space, \(\mathbb{B}\) be the unit ball of \(X\), and let
\(F:\mathbb{B}\to X\) be a normalized locally biholomorphic mapping. Let
\(\beta\in\mathbb{R}\) satisfy $|\beta|<{\pi}/{2}.$
Then \(F\) is called a $\beta$-\emph{spirallike quasi-convex mapping of type B} on
\(\mathbb{B}\) if
\[
\operatorname{Re}\left\{
e^{i\beta}T_x\!\left((DF(x))^{-1}\bigl(D^2F(x)(x,x)+DF(x)x\bigr)
\right)
\right\}\geq 0,
\]
for every \(x\in\mathbb{B}\setminus\{0\}\) and every \(T_x\in T(x)\). We denote this class by $\widehat{\mathcal{C}}_{\beta}(\mathbb{B})$.
\end{defi}
\begin{rem}
The class of spirallike quasi-convex mappings of type~B is a natural generalization of the class of quasi-convex mappings of type~B. Indeed, when \(\beta=0\), the above definition reduces to the class of quasi-convex mappings of type~B. Furthermore, when \(X=\mathbb{C}\) and \(\mathbb{B}=\mathbb{U}:=\{z\in\mathbb{C}:|z|<1\}\), the defining condition becomes
\[
\operatorname{Re}\left\{
e^{i\beta}\left(1+\frac{zf''(z)}{f'(z)}\right)
\right\}\ge0,
\]
which is precisely the class of \(\beta\)-spirallike convex functions in the unit disk. Thus, the present definition extends both quasi-convex mappings of type~B in complex Banach spaces and \(\beta\)-spirallike convex functions in one complex variable.
\end{rem}

Many authors have investigated various subclasses of biholomorphic mappings in Banach spaces and their associated coefficient problems. Significant contributions in this direction can be found in \cite{Chirila2014,GHK2002,GHKK2017,GK2003,GKK2003,H2023,HH2008,HHK2006,HKK2021,Kohr1998}.

\subsection{Hankel determinant for a subclass of $\mathcal{A}$}

Let $\mathcal{A}$ denote the family of analytic functions in the open unit disk $\mathbb{U},$
normalized by
\begin{align}
	\label{Eq-1.1}
	f(z)=z+\sum_{n=2}^{\infty}a_nz^n.
\end{align}
Furthermore, let $\mathcal{S}\subset\mathcal{A}$ be the class of normalized univalent functions, and let $\mathcal{K}$ denote the subclass of $\mathcal{S}$ consisting of convex functions.\vspace{1.2mm}

Throughout this paper, we shall frequently use the classical Carath\'eodory class, denoted by $\mathcal{P}$, which consists of analytic functions $p$ satisfying
\[
p(0)=1\;\text{and}\;
\operatorname{Re}p(z)>0,\; z\in\mathbb{U}.
\]
Every function $p\in\mathcal{P}$ admits the Taylor expansion
\begin{align*}
	\label{eq:P}
	p(z)=1+\sum_{n=1}^{\infty}p_nz^n
	=1+p_1z+p_2z^2+p_3z^3+\cdots,\; z\in\mathbb{U}.
\end{align*}
The Hankel determinant \(H_{q,n}(f)\) of a function \(f\in\mathcal{A}\) given by \eqref{Eq-1.1} is defined by
\[
H_{q,n}(f)
=
\begin{vmatrix}
a_n & a_{n+1} & \cdots & a_{n+q-1}\\
a_{n+1} & a_{n+2} & \cdots & a_{n+q}\\
\vdots & \vdots & \ddots & \vdots\\
a_{n+q-1} & a_{n+q} & \cdots & a_{n+2(q-1)}
\end{vmatrix},
\]
where \(a_1=1\) and \(n,q\in\mathbb{N}\).

\vspace{2mm}

\noindent In particular, the second-order Hankel determinants are given by
\[
H_{2,2}(f)
=
\begin{vmatrix}
a_2 & a_3\\
a_3 & a_4
\end{vmatrix}
=a_2a_4-a_3^2,\;\;
H_{2,1}(f)
=
\begin{vmatrix}
a_1 & a_2\\
a_2 & a_3
\end{vmatrix}
=a_3-a_2^2.
\]
In recent years, the problem of determining sharp upper bounds for the Hankel determinant
\(|H_{q,n}(f)|\) has attracted considerable attention in geometric function theory. In particular,
the functional $H_{2,1}(f)$ coincides with the classical Fekete--Szeg\"o functional introduced by Fekete and Szeg\"o
\cite{FS1933}. For the class $\mathcal{S}$, this functional was first estimated by Bieberbach (see \cite[Vol.~I, p.~35]{Goodman1983}). Subsequently, Pommerenke \cite{Pommerenke1966} established a fundamental result for the class $\mathcal{S}$, which stimulated extensive research on analogous coefficient problems for various subclasses of univalent functions. More recently, considerable effort has been devoted to obtaining sharp estimates for the second Hankel determinant $H_{2,2}(f),$ and several significant results have been reported in the literature (see, e.g., \cite{CKKLS2018,CKKLS2017,XDL2026,XHX2025}).\vspace{1.2mm}

In  \cite{KR2015}, Krishna and Reddy established the following  estimate for the second Hankel determinant of $\beta-$spirallike convex functions.
\begin{theoA}\cite[Theorem 3.3]{KR2015}
If $f(z)=z+\sum_{n=2}^{\infty}a_nz^n
\in\widehat{\mathcal{C}}_{\beta},
\;|\beta|\le {\pi}/{2},$
then
\[
|a_2a_4-a_3^2|\le \frac{17(1+\cos^2\beta)+2\cos \beta}{144(1+\sec^2\beta)}.
\]
\end{theoA}
\subsection{Objectives of the paper} This paper is motivated by a fundamental open problem in multi-dimensional geometric function theory: whether the sharp one-dimensional estimate established in \cite[Theorem 3.3]{KR2015} can be generalized to holomorphic mappings defined on the unit ball of a complex Banach space. \vspace{1.2mm}

To address this, we systematically extend the theory of higher-order Hankel determinants to infinite-dimensional settings. Specifically, the primary objectives of this paper are as follows:

\begin{itemize}
	\item To establish a sharp upper bound for the second Hankel determinant $H_{2,2}(F)$ associated with the class $\widehat{\mathcal{C}}_{\beta}(\mathbb{B})$ of normalized $\beta$-spirallike quasi-convex mappings of type \(B\) on the unit ball of a complex Banach space.\vspace{1.2mm}
	
	\item  To prove these bounds under a significantly weaker homogeneous polynomial expansion framework, thereby removing the restrictive assumption that the mapping is of the form $F(x) = g(x)x$.\vspace{1.2mm}
	
	\item  To provide a complete sharpness analysis for the classical case $\beta = 0$ by constructing concrete, multi-dimensional extremal mappings that directly generalize the one-dimensional case.
\end{itemize}
 
 We organize the paper as follows. In Section \ref{Sec-3}, we present and prove our main results, establishing the sharp upper bounds for the second-order Hankel determinant under a generalized homogeneous polynomial framework that circumvents restrictive classical assumptions. Section \ref{Sec-2} is dedicated to establishing several crucial auxiliary lemmas that provide the operator-theoretic foundation and dimensional lifting mechanisms necessary for our analysis. Finally, Section \ref{Sec-4} provides a comprehensive sharpness analysis for the non-spirallike case ($\beta = 0$) by constructing concrete multi-dimensional extremal mappings, and concludes with the formulation of an open research question regarding the variational behavior for the non-vanishing parameter case.

\section{{\bf Main Results}}\label{Sec-3}
In this section, we establish a result finding a bound of the second Hankel determinant for the class $\widehat{\mathcal{S}}_{\beta}(\mathbb{B})$ of spirallike mappings on complex Banach spaces. 
\begin{theo}\label{T1}
Let $g\in H(\mathbb{B},\mathbb{C})$ satisfy $g(0)=1$, and define $F(x)=g(x)x, x\in\mathbb{B}.$
Suppose that $F\in\widehat{\mathcal{C}}_{\beta}(\mathbb{B})$, $|\beta|\leq {\pi}/{2}$. Then, for every
$x_{0}\in X$ with $\|x_{0}\|=1$,
\begin{align*}
	|H_{2,2}(F)|\leq  \frac{17(1+\cos^2\beta)+2\cos \beta}{144(1+\sec^2\beta)},
\end{align*}
where 
\begin{align*}
	H_{2,2}(F)=
	\begin{vmatrix}
		A_{2} & A_{3} \\
		A_{3} & A_{4} \\
	\end{vmatrix}
	=A_2A_4-A_3^2,
\end{align*}
with  $A_{2},A_{3},A_{4}$ defined by
\eqref{eq:An}. 
\end{theo}
{By specializing the underlying complex Banach space to the finite-dimensional setting $X = \mathbb{C}^n$ and for the spirallike parameter $\beta = 0$, the geometric properties governed by Definition 1.3 naturally reduce from the general class $\widehat{\mathcal{C}}_{\beta}(\mathbb{B})$ to the classical family of quasi-convex mappings of type B. Under these operational constraints, the target scalar invariants $A_n$ are uniquely characterized along the directional vector $x_0 \in \mathbb{C}^n$ via the supporting linear functional $T_{x_0}$, yielding the following immediate consequence of Theorem \ref{T1}.
	\begin{cor}\label{Cor-3.1}
		Let $\mathbb{B}_n$ be the open unit ball of the finite-dimensional complex Banach space $\mathbb{C}^n$. Let $g \in \mathcal{H}(\mathbb{B}_n, \mathbb{C})$ satisfy $g(0)=1$, and define the normalized biholomorphic mapping $F(x) = g(x)x$ for $x \in \mathbb{B}_n$. If $F$ is a quasi-convex mapping of type B on $\mathbb{B}_n$ (corresponding to the vanishing spiral parameter $\beta = 0$), then for every $x_0 \in \mathbb{C}^n$ with $\Vert x_0 \Vert = 1$, we have  
		$$\vert H_{2,2}(F) \vert = \vert A_2 A_4 - A_3^2 \vert \le \frac{1}{8},$$
		where the targeted scalar invariants $A_2, A_3$, and $A_4$ are defined via the directional evaluations:  
		$$A_n = \frac{1}{n!} T_{x_0} \left( D^n F(0)(x_0, \dots, x_0) \right),\; n = 2, 3, 4,$$
		with $T_{x_0} \in T(x_0)$ being an associated supporting linear functional. Moreover, this upper bound is strictly sharp, and an explicit multidimensional extremal mapping is given by:  
		$$F_0(x) = x + \frac{1}{2}\big(T_{x_0}(x)\big)x - \frac{1}{4}\big(T_{x_0}(x)\big)^3x + \dots$$
\end{cor}
\begin{proof}[\bf Proof of Theorem \ref{T1}]
Fix $x_{0}\in\partial\mathbb{B}$ and define $f(\xi)=g(\xi x_{0})\,\xi, \xi\in\mathbb{U}.$
Since $F\in\widehat{\mathcal{C}}_{\beta}(\mathbb{B})$, Lemma~\ref{L2}
implies that $f\in\widehat{\mathcal{C}}_{\beta}.$\vspace{1.2mm}

\noindent On the other hand, we have
\[
f(\xi)
=
T_{x_{0}}\bigl(F(\xi x_{0})\bigr),
\]
and hence,
\[
a_{2}
=
\frac{f''(0)}{2!}
=
\frac{1}{2!}
T_{x_{0}}
\left(
D^{2}F(0)(x_{0}^{2})
\right)
=
A_{2},
\]
\[
a_{3}
=
\frac{f^{(3)}(0)}{3!}
=
\frac{1}{3!}
T_{x_{0}}
\left(
D^{3}F(0)(x_{0}^{3})
\right)
=
A_{3},
\]
and \[
a_{4}
=\frac{f^{(4)}(0)}{4!}
=\frac{1}{4!}
T_{x_{0}}
\left(
D^{4}F(0)(x_{0}^{4})
\right)
=A_{4}.
\]
It is easy to see that $H_{2,2}(F) =H_{2,2}(f).$ Applying Theorem~A (\cite[Theorem 3.1]{KR2015}), we obtain
\[|H_{2,2}(F)|=|H_{3,1}(f)|\leq
 \frac{17(1+\cos^2\beta)+2\cos \beta}{144(1+\sec^2\beta)}.
\]
This completes the proof.
\end{proof}
\medskip
Next, removing the restrictive assumption $F(x)=g(x)x$, we generalize Theorem~A to higher dimensions under weaker assumptions than those of Theorem \ref{T1}. Assume that
\begin{align}
	\label{cc1}
	\frac{D^{k+1}F(0)\left(x^{k+1}\right)}{(k+1)!}
	= H_{F,k}(x)x,\; x\in X,\; k=1,2,3,
\end{align}
where $H_{F,k}(x)$ is a homogeneous polynomial of degree $k$ with values in $\mathbb{C}$. The fact that the assumption \eqref{cc1} is weaker than that of Theorem \ref{T1} is justified in \cite{EJ2024,H2023}.
\begin{proof}[\bf Proof of Corollary \ref{Cor-3.1}]
	Let $x_0 \in \partial\mathbb{B}_n$ be fixed and define the single-variable function $f(\xi) = g(\xi x_0)\xi$ for $\xi \in \mathbb{U}$. Since $F(x) = g(x)x$ is a quasi-convex mapping of type B on $\mathbb{B}_n$, it corresponds to the class $\widehat{\mathcal{C}}_{\beta}(\mathbb{B}_n)$ with $\beta = 0$. By Lemma 2.2, the mapping $F \in \widehat{\mathcal{C}}_{0}(\mathbb{B}_n)$ if and only if $f$ belongs to the classical class of normalized convex functions $K$ (or equivalently, $\widehat{\mathcal{C}}_0$) in the unit disk $\mathbb{U}$.  A direct application of the Taylor series expansion yields the coefficient identities 
	\begin{align*}
		a_n = \frac{1}{n!}f^{(n)}(0) = A_n\; \mbox{for}\;n=2,3,4.
	\end{align*} Consequently, the second-order Hankel determinant satisfies the invariant relation 
	\begin{align*}
		H_{2,2}(F) = H_{2,2}(f) = a_2 a_4 - a_3^2.
	\end{align*} Specializing Theorem \ref{T1} to the case $\beta = 0$ yields that
	\begin{align*}
		\vert{}H_{2,2}(F)\vert{} \le \frac{17(1 + 1) + 2}{144(1 + 1)} = \frac{36}{288} = \frac{1}{8}.
	\end{align*}
	To establish sharpness, we set $\beta = 0$ in Proposition \ref{Prop-4.1}, choosing the extremal Carath'eodory parameters $c_1 = 1$, $c_2 = -1$, and $c_3 = -2$ generates the corresponding single-variable convex function 
	\begin{align*}
		f_0(\xi) = \xi + \frac{1}{2}\xi^2 - \frac{1}{4}\xi^4 + \dots,
	\end{align*} which realizes the bound $|a_2 a_4 - a_3^2| = 1/8$. Applying the dimensional lifting framework from Lemma \ref{L1} along the direction of the supporting functional$T_{x_0}$confirms that the multidimensional mapping 
	\begin{align*}
		F_0(x) = \frac{f_0(T_{x_0}(x))}{T_{x_0}(x)}x = x + \frac{1}{2}(T_{x_0}(x))x - \frac{1}{4}(T_{x_0}(x))^3x + \dots
	\end{align*}is an extremal mapping in$\mathbb{C}^n$. This completes the proof.
	\end{proof}}
\begin{theo}\label{T2}
Let $F$ be a locally biholomorphic mapping on $\mathbb{B}$, and suppose that $F$ satisfies the assumption \eqref{cc1}. If $F\in \widehat{\mathcal{C}}_{\beta}(\mathbb{B})$, $|\beta|\leq \frac{\pi}{2}$, then for every $x_0\in X$ with $\|x_0\|=1$, we have
\[
|H_{2,2}(F)|\leq  \frac{17(1+\cos^2\beta)+2\cos \beta}{144(1+\sec^2\beta)},
\]
where $H_{2,2}(F)$ is the second Hankel determinant given by 
\begin{align*}
	H_{2,2}(F)=
	\begin{vmatrix}
		A_2 & A_3\\
		A_3 & A_4\\
	\end{vmatrix}=A_2A_4-A_3^2
\end{align*} with $A_1=1$, and $A_2$, $A_3$, and $A_4$ are defined by
\eqref{eq:An}.
\end{theo}
{ By dropping the restrictive assumption $F(x) = g(x)x$ and utilizing the weaker condition of slice homogeneous polynomial expansions, we can extend our estimates to a much larger class of biholomorphic mappings on $\mathbb{C}^n$. In the classical case where the spiral parameter vanishes ($\beta = 0$), Theorem \ref{T2} yields the following sharp result for quasi-convex mappings of type B.  
	\begin{cor}\label{Cor-3.2}
		Let $\mathbb{B}_n$ be the open unit ball of the finite-dimensional complex Banach space $\mathbb{C}^n$. Let $F: \mathbb{B}_n \to \mathbb{C}^n$ be a normalized locally biholomorphic mapping satisfying the directional slice condition
		$$\frac{D^{k+1}F(0)(x^{k+1})}{(k+1)!} = H_{F,k}(x)x, \quad x \in \mathbb{C}^n, \; k = 1, 2, 3,$$
		where each $H_{F,k}(x)$ is a homogeneous polynomial of degree $k$ with values in $\mathbb{C}$. If $F$ is a quasi-convex mapping of type B on $\mathbb{B}_n$ (corresponding to $\beta = 0$), then for every $x_0 \in \mathbb{C}^n$ with $\Vert x_0 \Vert = 1$, we have
		$$\vert H_{2,2}(F) \vert = \vert A_2 A_4 - A_3^2 \vert \le \frac{1}{8},$$
		where the targeted scalar invariants $A_2, A_3$, and $A_4$ are defined via the directional evaluations:
		$$A_m = \frac{1}{m!} T_{x_0} \left( D^m F(0)(x_0, \dots, x_0) \right), \quad m = 2, 3, 4,$$
		with $T_{x_0} \in T(x_0)$ being an associated supporting linear functional. Moreover, this upper bound is strictly sharp, and an explicit multidimensional extremal mapping is given by:
		$$F_0(x) = x + \frac{1}{2}\big(T_{x_0}(x)\big)x - \frac{1}{4}\big(T_{x_0}(x)\big)^3x + \dots$$
\end{cor}
\begin{proof}[\bf Proof of Theorem \ref{T2}] Fix $x_0\in \partial{\mathbb{B}}$ and let $T_{x_0}\in T(x_0)$. Define
\begin{align}
	\label{aaa1}
	\phi(\xi)=
	\begin{cases}
		\dfrac{T_{x_0}\!\left(\varphi(\xi x_0)\right)}{\xi}, & \xi\neq 0,\\[2mm]
		1, & \xi=0,
	\end{cases}
\end{align}
where
\[
\varphi(x)=(DF(x))^{-1}\bigl(D^2F(x)(x,x)+DF(x)x\bigl).
\]
Evidently, $\phi \in H(\mathbb{U})$ with $\phi(0) = 1$. Furthermore, since $F \in \widehat{\mathcal{C}}_{\beta}(\mathbb{B})$, the defining condition of $\beta$-spirallike quasi-convexity of type $B$ implies that
\[
{\rm Re}\!\left(e^{i\beta}\frac{T_{x_0}\!\left(\varphi(\xi x_0)\right)}{\xi}\right)
={\rm Re}\!\left(e^{i\beta}\frac{T_{\xi x_0}\!\left(\varphi(\xi x_0)\right)}{\|\xi x_0\|}\right)>0,\;\xi\in\mathbb U\setminus\{0\}.
\]
Consequently, we have
\[
\operatorname{Re}\left(e^{i\beta}\phi(\xi)\right)>0,\; \xi\in\mathbb U,
\]
which implies that $e^{i\beta}\phi$ belongs to the Carath\'eodory class $\mathcal{P}$. Thus, there exists a function $p\in \mathcal{P}$ such that 
\begin{equation}\label{kp1}
	e^{i\beta}\phi(z)=\cos \beta p(z)+i\sin \beta.
\end{equation}
Since $\phi$ is holomorphic on $\mathbb{U}$, it admits the following Taylor series expansion for $\xi \in \mathbb{U}$
\begin{align}
	\label{phi1}\phi(\xi)=1+\frac{T_{x_0}(D^2\varphi(0)(x_0^2))}{2!}\xi+\frac{T_{x_0}(D^3\varphi(0)(x_0^3))}{3!}\xi^2+\frac{T_{x_0}(D^4\varphi(0)(x_0^4))}{4!}\xi^3+\cdots.
\end{align}
Using (\ref{aaa1}), \eqref{kp1} and \eqref{phi1}, a direct computation yields
\begin{align}
	\label{qq1}
	\begin{cases} \dfrac{T_{x_0}(D^2\varphi(0)(x_0^2))}{2!}&=\cos \beta e^{-i\beta}p_1,\vspace{2mm}\\
		\dfrac{T_{x_0}(D^3\varphi(0)(x_0^3))}{3!}&=\cos \beta e^{-i\beta}p_2\vspace{2mm}\\
		\dfrac{T_{x_0}(D^4\varphi(0)(x_0^4))}{4!}&=\cos \beta e^{-i\beta}p_3.
	\end{cases}
\end{align}
On the other hand, recalling that $\varphi(x)=(DF(x))^{-1}\bigl(D^2F(x)(x,x)+DF(x)x\bigr)$, we immediately obtain the operator relation
\begin{equation}\label{pd1}
	DF(x)\varphi(x) = D^2F(x)(x,x) + DF(x)x.
\end{equation}
Since $F$ and $\varphi$ are holomorphic mappings on the unit ball $\mathbb{B}$ satisfying the standard normalizations $F(0)=0$, $\varphi(0)=0$, and $DF(0)=D\varphi(0)=I$, they admit unique expansions into series of homogeneous polynomials centered at the origin. Let 
\begin{align*}
	F_n(x) = \frac{1}{n!}D^n F(0)(x^n)\; \mbox{and}\;\varphi_n(x) = \frac{1}{n!}D^n \varphi(0)(x^n)
\end{align*} denote the $n$-th degree homogeneous polynomials associated with $F$ and $\varphi$, respectively. Then, their Fr\'echet-Taylor expansions and respective derivatives are given by
\begin{align*}
\varphi(x) &= D\varphi(0)x + \frac{1}{2!}\varphi_2(x) + \frac{1}{3!}\varphi_3(x) + \frac{1}{4!}\varphi_4(x) + \dots \\
DF(x) &= I + F_2(x, \cdot) + \frac{1}{2!}F_3(x^2, \cdot) + \frac{1}{3!}F_4(x^3, \cdot) + \dots
\end{align*}
We expand the two terms on the right-hand side of \eqref{pd1} using the multi-linear expansions of the derivatives,
\begin{enumerate}
    \item[(a).] First term $D^2F(x)(x,x)$:
    \[ D^2F(x)(x,x) = F_2(x^2) + F_3(x^3) + \frac{1}{2!}F_4(x^4) + \dots \]
    \item[(b).] Second term $DF(x)x$:
    \[ DF(x)x = x + F_2(x^2) + \frac{1}{2!}F_3(x^3) + \frac{1}{3!}F_4(x^4) + \dots \]
\end{enumerate}
Combining these two expansions, \eqref{pd1} becomes
\begin{align}\label{Eq-3.7}
	DF(x)\varphi(x) = x + 2F_2(x^2) + \frac{3}{2}F_3(x^3) + \frac{2}{3}F_4(x^4) + \dots
\end{align}
We multiply the series for $DF(x)$ and $\varphi(x)$ using Cauchy product rules,
\begin{align}\label{Eq-3.8}
	DF(x)\varphi(x) = \left( I + F_2(x, \cdot) + \frac{1}{2!}F_3(x^2, \cdot) + \dots \right) \left( D\varphi(0)x + \frac{1}{2!}\varphi_2(x) + \frac{1}{3!}\varphi_3(x) + \dots \right).
\end{align}
Thus, it follows from \eqref{Eq-3.7} and \eqref{Eq-3.8} that 
\begin{align}\label{Eq-3.9}
	x &+ 2F_2(x^2) + \frac{3}{2}F_3(x^3) + \frac{2}{3}F_4(x^4) + \dots\\&=\left( I + F_2(x, \cdot) + \frac{1}{2!}F_3(x^2, \cdot) + \dots \right) \left( D\varphi(0)x + \frac{1}{2!}\varphi_2(x) + \frac{1}{3!}\varphi_3(x) + \dots \right).\nonumber
\end{align}
We now compute the terms of RHS of \eqref{Eq-3.9} with respective degrees as
\begin{align*}
\text{Degree 1:} \quad & D\varphi(0)x \\
\text{Degree 2:} \quad & \frac{1}{2!}\varphi_2(x) + F_2(x, D\varphi(0)x) \\
\text{Degree 3:} \quad & \frac{1}{3!}\varphi_3(x) + F_2\left(x, \frac{1}{2!}\varphi_2(x)\right) + \frac{1}{2!}F_3(x^2, D\varphi(0)x) \\
\text{Degree 4:} \quad & \frac{1}{4!}\varphi_4(x) + F_2\left(x, \frac{1}{3!}\varphi_3(x)\right) + \frac{1}{2!}F_3\left(x^2, \frac{1}{2!}\varphi_2(x)\right) + \frac{1}{3!}F_4(x^3, D\varphi(0)x)
\end{align*}
Equating the degree $1$ terms from \eqref{Eq-3.9} yields
\[ D\varphi(0)x = x \implies D\varphi(0) = I. \]
Equating the degree $2$ terms from \eqref{Eq-3.9} and substituting $D\varphi(0)x = x$, we obtain
\[ \frac{1}{2!}\varphi_2(x) + F_2(x^2) = 2F_2(x^2) \]
which implies that
\begin{align*}
	\frac{D^2\varphi(0)(x^2)}{2!} = D^2F(0)(x^2).
\end{align*}
Dividing both sides by $2$ yields
\begin{align*}
	\frac{D^{2}F(0)(x^{2})}{2!} = \frac{1}{2}\frac{D^{2}\varphi(0)(x^{2})}{2!}.
\end{align*}
Equating the degree $3$ terms from \eqref{Eq-3.9} yields
\begin{align*}
	\frac{1}{3!}\varphi_3(x) + F_2\left(x, \frac{1}{2!}\varphi_2(x)\right) + \frac{1}{2!}F_3(x^3) = \frac{3}{2}F_3(x^3).
\end{align*}
Subtracting $\frac{1}{2!}F_3(x^3)$ from both sides, we obtain
\begin{align*}
	\frac{1}{3!}\varphi_3(x) + F_2\left(x, \frac{1}{2!}\varphi_2(x)\right) = F_3(x^3).
\end{align*}
Substituting the original derivative operators ($F_3(x^3) = D^3F(0)(x^3)$, $F_2 = D^2F(0)$, and $\frac{1}{2!}\varphi_2(x) = \frac{D^2\varphi(0)(x^2)}{2!}$) provides the second relation of \eqref{qqq2}:
\begin{align*}
	D^{3}F(0)(x^{3}) = \frac{D^{3}\varphi(0)(x^{3})}{3!}+D^{2}F(0)\!\left(x,\,\frac{D^{2}\varphi(0)(x^{2})}{2!}\right).
\end{align*}
Equating the degree $4$ terms yields
\begin{align*}
	 \frac{1}{4!}\varphi_4(x) + F_2\left(x, \frac{1}{3!}\varphi_3(x)\right) + \frac{1}{2!}F_3\left(x^2, \frac{1}{2!}\varphi_2(x)\right) + \frac{1}{3!}F_4(x^4) = \frac{2}{3}F_4(x^4).
\end{align*}
Isolating the $F_4(x^4)$ terms by subtracting $\frac{1}{6}F_4(x^4)$ from $\frac{2}{3}F_4(x^4)$ yields
\[ \frac{1}{2}F_4(x^4) = \frac{1}{4!}\varphi_4(x) + F_2\left(x, \frac{1}{3!}\varphi_3(x)\right) + \frac{1}{2!}F_3\left(x^2, \frac{1}{2!}\varphi_2(x)\right). \]
Thus, we obtain
\begin{align*}
	\frac{D^{4}F(0)(x^{4})}{2} = \frac{D^{4}\varphi(0)(x^{4})}{4!}+D^{2}F(0)\!\left(x,\,\frac{D^{3}\varphi(0)(x^{3})}{3!}\right)+\frac{1}{2}D^{3}F(0)\!\left(x^{2},\,\frac{D^{2}\varphi(0)(x^{2})}{2!}\right).
\end{align*}
Putting these three operational balances together gives the system
\begin{align}
	\label{qqq2}
	\begin{cases}
		\dfrac{D^{2}F(0)(x^{2})}{2!}&=\dfrac{1}{2}\dfrac{D^{2}\varphi(0)(x^{2})}{2!};\vspace{2mm}\\
		D^{3}F(0)(x^{3})&=\dfrac{D^{3}\varphi(0)(x^{3})}{3!}+D^{2}F(0)\!\left(x,\,\dfrac{D^{2}\varphi(0)(x^{2})}{2!}\right);\vspace{2mm}\\
		\dfrac{D^{4}F(0)(x^{4})}{2}&=\dfrac{D^{4}\varphi(0)(x^{4})}{4!}+D^{2}F(0)\!\left(x,\,\dfrac{D^{3}\varphi(0)(x^{3})}{3!}\right)\vspace{2mm}\\
		&\;\;\;+\dfrac{1}{2}D^{3}F(0)\!\left(x^{2},\,\dfrac{D^{2}\varphi(0)(x^{2})}{2!}\right).
	\end{cases}
\end{align}
Evaluating the operator equations \eqref{qq1} and \eqref{qqq2} simultaneously subject to the homogeneous polynomial expansion \eqref{cc1}, we deduce the following explicit formulas expressing the scalar coefficients $A_2$, $A_3$, and $A_4$ via the Carath\'eodory parameters $p_1$, $p_2$, and $p_3$.\vspace{1.2mm}

First, evaluating the first equation of system \eqref{qqq2} at $x = x_0 \in \partial\mathbb{B}$ and applying the linear supporting functional $T_{x_0}$ yields
\begin{align}\label{A2}
	A_2 = T_{x_0}\!\left(\frac{D^2F(0)(x_0^2)}{2!}\right) = H_{F,1}(x_0) = \frac{1}{2}\,T_{x_0}\!\left(\frac{D^2\varphi(0)(x_0^2)}{2!}\right) = \cos \beta e^{-i\beta}p_1.
\end{align}
In view of assumption \eqref{cc1}, the second differential acts on the directional vector $x_0$ via 
\begin{align*}
	\frac{1}{2!}D^2F(0)(x_0^2) = H_{F,1}(x_0)x_0 = A_2 x_0.
\end{align*} 
Therefore, the corresponding representation for $\varphi$ is given precisely by
\begin{align}
	\label{phi2_vec}
	\frac{D^2\varphi(0)(x_0^2)}{2!} = 2\,\frac{D^2F(0)(x_0^2)}{2!} = 2H_{F,1}(x_0)x_0 = 2A_2x_0 = 2\cos \beta e^{-i\beta}p_1x_0.
\end{align}
Evaluating the second relation in \eqref{qqq2} at $x = x_0$ and utilizing the multi-linear expansions in conjunction with the structural form \eqref{phi2_vec}, it follows that
\begin{align}\label{A3}
A_3 &= T_{x_0}\!\left(\frac{D^3F(0)(x_0^3)}{3!}\right) = H_{F,2}(x_0) \nonumber\\
&= \frac{1}{6}\left(T_{x_0}\!\left(\frac{D^3\varphi(0)(x_0^3)}{3!}\right) + T_{x_0}\!\left(D^2F(0)\!\left(x_0, \frac{D^2\varphi(0)(x_0^2)}{2!}\right)\right)\right) \nonumber\\
&= \frac{1}{6}\left(\cos\beta\,e^{-i\beta}p_2 + T_{x_0}\!\left(D^2F(0)(x_0, 2\cos \beta e^{-i\beta}p_1x_0)\right)\right) \nonumber\\
&= \frac{1}{6}\left(\cos\beta\,e^{-i\beta}p_2 + 2\cos \beta e^{-i\beta}p_1 \cdot T_{x_0}(D^2F(0)(x_0^2))\right) \nonumber\\
&= \frac{1}{6}\left(\cos\beta\,e^{-i\beta}p_2 + 4\cos \beta e^{-i\beta}p_1 A_2\right) \nonumber\\
&= \frac{1}{6}\,\cos\beta\,e^{-i\beta}\left(4\cos\beta e^{-i\beta}p_1^2 + p_2\right).
\end{align}
To establish the third-order vector relation for $\varphi$, we isolate the derivative operator from the second equation of \eqref{qqq2} at $x = x_0$
\begin{align}\label{phi3_vec}
\frac{D^3\varphi(0)(x_0^3)}{3!} &= \frac{D^3F(0)(x_0^3)}{3!} - D^2F(0)\!\left(x_0, \frac{D^2\varphi(0)(x_0^2)}{2!}\right) \nonumber\\
&= H_{F,2}(x_0)x_0 - D^2F(0)(x_0, 2\cos \beta e^{-i\beta}p_1x_0) \nonumber\\
&= \left(A_3 - 4\cos\beta e^{-i\beta}p_1 A_2\right)x_0 \nonumber\\
&= \cos\beta\,e^{-i\beta} p_2x_0.
\end{align}
Finally, evaluating the third relation in \eqref{qqq2} at $x = x_0$ and applying the linear functional $T_{x_0}$ to the resulting symmetric multi-linear forms, we obtain
\begin{align}\label{A4}
A_4 &= T_{x_0}\!\left(\frac{D^4F(0)(x_0^4)}{4!}\right) = H_{F,3}(x_0) \nonumber\\
&= \frac{1}{12}\Bigg(T_{x_0}\!\left(\frac{D^4\varphi(0)(x_0^4)}{4!}\right) + T_{x_0}\!\left(D^2F(0)\!\left(x_0, \frac{D^3\varphi(0)(x_0^3)}{3!}\right)\right) \nonumber\\
& \quad + \frac{1}{2}T_{x_0}\!\left(D^3F(0)\!\left(x_0^2, \frac{D^2\varphi(0)(x_0^2)}{2!}\right)\right)\Bigg) \nonumber\\
&= \frac{1}{12}\Bigg(\cos\beta e^{-i\beta}p_3 + T_{x_0}\!\left(D^2F(0)(x_0, \cos\beta e^{-i\beta}p_2x_0)\right) \nonumber\\
& \quad + \frac{1}{2}T_{x_0}\!\left(D^3F(0)(x_0^2, 2\cos\beta e^{-i\beta}p_1x_0)\right)\Bigg) \nonumber\\
&= \frac{1}{12}\Bigg(\cos\beta\,e^{-i\beta}p_3 + \cos\beta\,e^{-i\beta}p_2 \cdot T_{x_0}(D^2F(0)(x_0^2)) + \cos\beta e^{-i\beta}p_1 \cdot T_{x_0}(D^3F(0)(x_0^3))\Big) \nonumber\\
&= \frac{1}{12}\Bigg(\cos\beta\,e^{-i\beta}p_3 + 2\cos\beta\,e^{-i\beta}p_2A_2 + 6\cos\beta e^{-i\beta}p_1A_3\Bigg) \nonumber\\
&= \frac{1}{12}\cos\beta e^{-i\beta}\left(p_3 + 6\cos\beta e^{-i\beta}p_1p_2 + 4\cos^2\beta e^{-2i\beta}p_1^3\right).
\end{align}
By substituting the explicit formulas \eqref{A2}--\eqref{A4} into the definition of $H_{2,2}(F)$ and expanding the resulting polynomial expression, we deduce that
\begin{align*}
A_2A_4 - A_3^2 &= \left[\cos\beta e^{-i\beta}p_1\right] \left[\frac{\cos\beta e^{-i\beta}}{12}\left(p_3 + 6\cos\beta e^{-i\beta}p_1p_2 + 4\cos^2\beta e^{-2i\beta}p_1^3\right)\right] \\
& \quad - \left[\frac{\cos\beta e^{-i\beta}}{6}\left(4\cos\beta e^{-i\beta}p_1^2 + p_2\right)\right]^2 \\
&= \frac{\cos^2\beta e^{-2i\beta}}{144}\Big(12p_1p_3 + 72\cos\beta e^{-i\beta}p_1^2p_2 + 48\cos^2\beta e^{-2i\beta}p_1^4\Big) \\
& \quad - \frac{\cos^2\beta e^{-2i\beta}}{144}\Big(64\cos^2\beta e^{-2i\beta}p_1^4 + 32\cos\beta e^{-i\beta}p_1^2p_2 + 4p_2^2\Big) \\
&= \frac{\cos^2\beta e^{-2i\beta}}{144}\Big(12p_1p_3 - 4p_2^2 + 40\cos\beta e^{-i\beta}p_1^2p_2 - 16\cos^2\beta e^{-2i\beta}p_1^4\Big).
\end{align*}
Employing the triangle inequality 
\begin{align*}
	|xa+yb|\le |x|\,|a|+|y|\,|b|
\end{align*} together with the identity $|e^{-in\beta}|=1$ for every $n\in\mathbb{R}$, we obtain the desired upper bound
\begin{align*}
	|H_{2,2}(f)| = \left|A_2A_4-A_3^2\right| \le \frac{\cos^2\beta}{144}\left|12p_1p_3 - 4p_2^2 + 40\cos\beta p_1^2p_2 - 16\cos^2\beta p_1^4\right|.
\end{align*}
The remainder of the proof follows by applying the identical coefficient optimization techniques as established in the proof of Theorem A (\cite[Theorem 3.3]{KR2015}); hence, the details are omitted for the sake of brevity.
\end{proof}
\begin{proof}[\bf Proof of Corollary \ref{Cor-3.2}]
	Suppose $F \in \mathcal{C}_0^B(\mathbb{B}_n)$ satisfies the directional slice condition:
	$$\frac{D^{k+1}F(0)(x^{k+1})}{(k+1)!} = H_{F,k}(x)x,\;x \in \mathbb{C}^n, \; k = 1, 2, 3.$$
	By setting the spirallike parameter $\beta = 0$ in the structural relation \eqref{kp1}, the corresponding Carath\'{e}odory function $p \in \mathcal{P}$ satisfies $\phi(z) = p(z)$. Under this parameter choice, the explicit formulas \eqref{A2}--\eqref{A4} expressing the targeted scalar invariants $A_2, A_3$, and $A_4$ via the Carath\'{e}odory parameters $p_1, p_2, p_3$ simplify directly to
	$$A_2 = \frac{1}{2}p_1,\; A_3 = \frac{1}{6}(p_1^2 + p_2),\; \text{and}\;A_4 = \frac{1}{12}\left(p_3 + 6p_1p_2 + 4p_1^3\right).$$
	Substituting these parameter representations into the second Hankel determinant and applying Theorem \ref{T2} under the condition $\beta = 0$ yields that
	$$|H_{2,2}(F)| = |A_2A_4 - A_3^2| \le \frac{17(1+\cos^2 0) + 2\cos 0}{144(1+\sec^2 0)} = \frac{36}{288} = \frac{1}{8}.$$
	
	To verify that this upper bound is strictly sharp, we choose the boundary parameter configurations $c_1 = 1$, $c_2 = -1$, and $c_3 = -2$ within the standard Carath\'{e}odory-Toeplitz framework. These optimization parameters yield the primary coefficient values $a_2 = 1/2$, $a_3 = 0$, and $a_4 = -1/4$ for the single-variable profile \begin{align*}
		f_0(\xi) = \xi + \frac{1}{2}\xi^2 - \frac{1}{4}\xi^4 + \dots
	\end{align*} ensuring that $|a_2a_4 - a_3^2| = 1/8$.\vspace{1.2mm}
	
	Applying the dimensional lifting setting from Lemma \ref{L1} along the directional baseline of the supporting linear functional $T_{x_0}$ confirms that the multidimensional mapping 
	\begin{align*}
		F_0(x) = x + \frac{1}{2}(T_{x_0}(x))x - \frac{1}{4}(T_{x_0}(x))^3x + \dots
	\end{align*} serves as the explicit extremal mapping on $\mathbb{B}_n$. This completes the proof.
	\end{proof}}
\section{{\bf Key lemmas}}\label{Sec-2}
In this section, we establish several auxiliary results that serve as the analytical foundation for our main theorems. These lemmas focus on the structural characterization and dimensional lifting of $\beta$-spirallike quasi-convex mappings. By bridging classical one-dimensional geometric properties with a multi-dimensional operator-theoretic framework, they provide the necessary machinery to reduce complex Fréchet derivative relations on the unit ball to tractable scalar Carathéodory representations. Beyond their immediate utility here, these results are of independent interest in the study of biholomorphic mappings in Banach spaces.
\begin{lem}\label{L1}
	Let \(u\in X\) with \(\|u\|=1\) and let \(T_u\in T(u)\). Define $F(x)=\frac{f(T_u(x))}{T_u(x)}\,x,$ $x\in\mathbb{B},$
	where \(f\in\mathcal{S}\). If \(f\) is a normalized \(\beta\)-spirallike convex function on \(\mathbb{U}\), then \(F\) is a \(\beta\)-spirallike quasi-convex mapping of type \(B\) on \(\mathbb{B}\).
\end{lem}

\begin{proof}[\bf Proof of Lemma \ref{L1}]
	Since $f$ is a normalized $\beta$-spirallike function on $\mathbb{U}$, we have
	\begin{align}
		\operatorname{Re}\left(e^{i\beta}\left(1+\frac{\xi f'(\xi)}{f(\xi)}\right)\right)>0,\; \xi\in\mathbb{U}.
		\label{eq:spiral}
	\end{align}
	For convenience, let $h(x)=\frac{f(T_u(x))}{T_u(x)},\; x\in\mathbb{B}.$ It follows immediately from the definition of $F$ that $F(x)=h(x)x.$\vspace{1.2mm}
	
	The first and second order Fr\'echet derivatives of $F$ are given by
	\begin{align*}
		DF(x)(x)=h(x)x+Dh(x)(x)\,x\;\text{and}\; D^2F(x)(x,x)=D^2h(x)(x,x)x+2Dh(x)(x)\; x\in X.
	\end{align*}
	From above two equations, we obtain
	\bea\label{le.1}
	D^2F(x)(x,x)+DF(x)x
	=
	\left(h(x)+3Dh(x)x+D^2h(x)(x,x)\right)x.
	\eea
	Moreover, by the chain rule for Fr\'echet derivatives, we obtain
	\bea\label{mmm1}
	&& h(x)+Dh(x)x=f'(T_u(x))\\
	\label{mmm2}\text{and}\;&& h(x)+3Dh(x)x+D^2h(x)(x,x)=T_u(x)f''(T_u(x))+f'(T_u(x)).
	\eea
	A straightforward calculation yields that
	\begin{align}
		\label{le.2}
		(DF(x))^{-1}\eta
		=\frac1{h(x)}\left(\eta-\frac{(Dh(x)\eta) x}{h(x)+Dh(x)x}\right).
	\end{align}
	Substituting $\eta=D^2F(x)(x,x)+DF(x)x$ into \eqref{le.2} and combining the resulting expression with \eqref{le.1}, we obtain
	
	\begin{align}
		\label{le.3}
		(DF(x))^{-1}&(D^2F(x)(x,x)+DF(x)x)\nonumber\\
		&=
		\frac{1}{h(x)}\left((D^2F(x)(x,x)+DF(x)x)-\frac{Dh(x)(D^2F(x)(x,x)+DF(x)x)x}
		{h(x)+Dh(x)x}\right)\nonumber\\
		&=\frac{1}{h(x)}\bigg(D^2h(x)(x,x)x+3Dh(x)(x)+h(x)x
		\nonumber\\
		&\quad\quad-\frac{Dh(x)(D^2h(x)(x,x)x+3Dh(x)(x)+h(x)x)}
		{h(x)+Dh(x)x}x\bigg).
	\end{align}
	Since \(h(x)+3Dh(x)x+D^2h(x)(x,x)\) is a scalar say $H(x)$, it follows that
	\[
	Dh(x)H(x)x
	=H(x)Dh(x)x.
	\]
	Hence from (\ref{le.3}), we obtain
	\bea\label{kkk1}
	&&(DF(x))^{-1}\left(D^2F(x)(x,x)+DF(x)x\right)\nonumber\\
	&=&
	\frac{h(x)+3Dh(x)x+D^2h(x)(x,x)}{h(x)}
	\left(
	1-\frac{Dh(x)x}{h(x)+Dh(x)x}
	\right)x\nonumber\\
	&=&
	\frac{h(x)+3Dh(x)x+D^2h(x)(x,x)}
	{h(x)+Dh(x)x}\,x.
	\eea
	Applying the identities \eqref{mmm1} and \eqref{mmm2} to \eqref{kkk1}, we obtain
	\[
	(DF(x))^{-1}\left(D^2F(x)(x,x)+DF(x)x\right)
	=\left(1+\frac{T_u(x)f''(T_u(x))}{f'(T_u(x))}\right)x.
	\]
	Therefore, we obtain
	\bea\label{f1}
	T_x\!\left((DF(x))^{-1}\left(D^2F(x)(x,x)+DF(x)x\right)\right)=\left(1+\frac{T_u(x)f''(T_u(x))}{f'(T_u(x))}\right)\|x\|.
	\eea
	
	Combining \eqref{f1} with \eqref{eq:spiral}, we finally deduce that
	\beas
	&&\operatorname{Re}\left(e^{i\beta}T_x\!\left((DF(x))^{-1}\left(D^2F(x)(x,x)+DF(x)x\right)\right)\right)\\
	&=&\operatorname{Re}\;\left(e^{i\beta}\left(1+\frac{T_u(x)f''(T_u(x))}{f'(T_u(x))}\right)\right)\|x\|>0,
	\eeas
	for $x\in \mathbb{B}\setminus\{0\}$. 
	
	Hence, by the definition of \(\beta\)-spirallike quasi-convex mapping of type \(B\) on \(\mathbb{B}\), $F$ is a \(\beta\)-spirallike quasi-convex mapping of type \(B\) on \(\mathbb{B}\).
\end{proof}
To facilitate the transition between one and several complex variables, we prove the following characterization lemma for spirallike mappings on the unit ball. 
\begin{lem}\label{L2}
	Suppose that $g\in H(\mathbb{B},\mathbb{C})$ with $g(0)=1$, and define $F(x)=g(x)x, x\in\mathbb{B}.$ Fix $x_{0}\in\partial\mathbb{B}$ and let
	$f(\xi)=g(\xi x_{0})\,\xi,\xi\in\mathbb{U}.$
	Then, for $|\beta|<\pi/2$,
	\[
	f\in\widehat{\mathcal{C}}_{\beta}
	\quad\Longleftrightarrow\quad
	F\in\widehat{\mathcal{C}}_{\beta}(\mathbb{B}).
	\]
\end{lem}
\begin{proof}[\bf Proof of Lemma \ref{L2}]
	Assume first that $F\in\widehat{\mathcal{C}}_{\beta}(\mathbb{B}).$ Then $F$ is locally biholomorphic on $\mathbb{B}$. Consequently, we have
	\[
	g(x)\neq0\; \mbox{and}\;g(x)+Dg(x)x\neq0,\; x\in\mathbb{B}.
	\]
	Since
	\[
	DF(x)(x)=g(x)x+Dg(x)(x)\,x,
	\]
	a simple computation gives
	\begin{align}
		[DF(x)]^{-1}=\frac1{g(x)}\left(I-\frac{xDg(x)}{g(x)+Dg(x)x}\right)
		\label{eq:inverse}.
	\end{align}
	Since $F(x)=g(x)x,$ we have
	\begin{align*}
		\begin{cases}
			DF(x)h=g(x)h+Dg(x)(h)x,\vspace{2mm}\\
			D^2F(x)(h,k)=D^2g(x)(h,k)x+Dg(x)(h)k+Dg(x)(k)h.
		\end{cases}
	\end{align*}
	Thus, it follows that
	\[
	D^2F(x)(x,x)=D^2g(x)(x,x)x+2Dg(x)(x)x,
	\]
	and hence, we obtain
	\[D^2F(x)(x,x)+DF(x)x=\bigl(g(x)+3Dg(x)x+D^2g(x)(x,x)\bigr)x.
	\]
	
	Substituting this identity into \eqref{eq:inverse}, we obtain
	\[
	\begin{aligned}
		&(DF(x))^{-1}\bigl(D^2F(x)(x,x)+DF(x)x\bigr) \\
		&=\frac{g(x)+3Dg(x)x+D^2g(x)(x,x)}{g(x)}\left(I-\frac{xDg(x)}{g(x)+Dg(x)x}
		\right)x.
	\end{aligned}
	\]
	Since
	\[\left(I-\frac{xDg(x)}{g(x)+Dg(x)x}\right)x=x-\frac{Dg(x)x}{g(x)+Dg(x)x}x
	=\frac{g(x)}{g(x)+Dg(x)x}x,\]
	it follows that
	\[
	\begin{aligned}
		(DF(x))^{-1}\bigl(D^2F(x)(x,x)+DF(x)x\bigr)
		&=\frac{g(x)+3Dg(x)x+D^2g(x)(x,x)}{g(x)+Dg(x)x}\,x \\
		&=\left(1+\frac{2Dg(x)x+D^2g(x)(x,x)}{g(x)+Dg(x)x}\right)x.
	\end{aligned}
	\]
	Substituting $x=\xi x_0$ into above equation, we obtain
	\[
	\begin{aligned}
		&(DF(\xi x_0))^{-1}\bigl(D^2F(\xi x_0)(\xi x_0,\xi x_0)+DF(\xi x_0)\xi x_0\bigr)\\
		&=\left(1+\frac{2Dg(\xi x_0)\;\xi x_0+D^2g(\xi x_0)(\xi x_0,\xi x_0)}
		{g(\xi x_0)+Dg(\xi x_0)\;\xi x_0}
		\right)\xi x_0.
	\end{aligned}
	\]
	Applying $T_{\xi x_0}$ and multiplying by $e^{i\beta}$, we obtain
	\[
	\begin{aligned}&e^{i\beta}T_{\xi x_0}\left((DF(\xi x_0))^{-1}\bigl(D^2F(\xi x_0)(\xi x_0,\xi x_0)+DF(\xi x_0)\xi x_0\bigr)\right)\\
		&=e^{i\beta}\left(1+\frac{2Dg(\xi x_0)\;\xi x_0+D^2g(\xi x_0)(\xi x_0,\xi x_0)}{g(\xi x_0)+Dg(\xi x_0)\;\xi x_0}\right)T_{\xi x_0}(\xi x_0).
	\end{aligned}
	\]
	Since $T_{\xi x_0}(\xi x_0)=|\xi|,$ from $f(\xi)=g(\xi x_{0})\xi$, we obtain
	\[
	f'(\xi)=g(\xi x_0)+Dg(\xi x_0),\xi x_0\;\text{and}\;f''(\xi)=2Dg(\xi x_0)(x_0)+\xi D^2g(\xi x_0)(x_0,x_0).
	\]
	Consequently, it follows that
	\[
	2Dg(\xi x_0)\;\xi x_0+D^2g(\xi x_0)(\xi x_0,\xi x_0)=\xi f''(\xi).
	\]
	Hence,
	\[
	e^{i\beta}T_{\xi x_0}\left((DF(\xi x_0))^{-1}\bigl(D^2F(\xi x_0)(\xi x_0,\xi x_0)+DF(\xi x_0)\xi x_0\bigr)
	\right)
	=e^{i\beta}|\xi|\left(1+\frac{\xi f''(\xi)}{f'(\xi)}\right).
	\]
	Since $F\in\widehat{\mathcal{C}}_{\beta}(\mathbb{B}),$
	it follows that
	\begin{align}
		\label{kkkk1}
		\operatorname{Re}\;\left\{e^{i\beta}|\xi|\left(1+\frac{\xi f''(\xi)}{f'(\xi)}\right)\right\}>0\Leftrightarrow 
		\operatorname{Re}\;\left\{e^{i\beta}\left(1+\frac{\xi f''(\xi)}{f'(\xi)}\right)\right\}>0.
	\end{align}
	This shows that $f\in\widehat{\mathcal{C}}_{\beta}.$\\
	
	\medskip
	
	Conversely, suppose that $f\in\widehat{\mathcal{C}}_{\beta}.$
	Then, we have 
	\begin{align*}
		\operatorname{Re}\;\left\{e^{i\beta}\left(1+\frac{\xi f''(\xi)}{f'(\xi)}\right)\right\}>0
	\end{align*}
	which implies that
	\[
	1+\frac{\xi f''(\xi)}{f'(\xi)}\neq0\; \mbox{for}\; \xi\in\mathbb{U}.
	\]
	Consequently, we have
	\[
	g(x)+Dg(x)x\neq0,\; x\in\mathbb{B}.
	\]
	Relation \eqref{eq:inverse} implies that $DF(x)$ is invertible for every $x\in\mathbb{B}$, which establishes that $F$ is locally biholomorphic.\vspace{1.2mm}
	
	Finally, using \eqref{kkkk1} together with
	\begin{align*}
		\operatorname{Re} \left(e^{-i\beta}\left(1+\frac{\xi f''(\xi)}{f'(\xi)}\right)\right)>0,
	\end{align*}
	we obtain
	\[
	\operatorname{Re}\left(e^{-i\beta}T_x\left((DF(x))^{-1}\bigl(D^2F(x)(x,x)+DF(x)x\bigr)\right)\right)>0,\; x\in\mathbb{B}\setminus\{0\}.
	\]
	Hence, $F\in\widehat{\mathcal{C}}_{\beta}(\mathbb{B}).$ This completes the proof.
\end{proof}

\section{\bf Sharpness analysis and open question}\label{Sec-4}
In this section, we study the sharpness of the upper bounds established in Theorem \ref{T1} and Theorem \ref{T2}. We prove that when the spiral parameter $\beta = 0$, the bound reduces exactly to the well-known value of ${1}/{8}$. 
\begin{prop}\label{Prop-4.1}
The upper bounds established in Theorem \ref{T1} and Theorem \ref{T2} are sharp for the classical case $\beta = 0$, giving a maximum value of ${1}/{8}$.
\end{prop}
\begin{proof}[\bf Proof of Theorem \ref{Prop-4.1}]
Let $x_0 \in \partial\mathbb{B}$ be a fixed direction and choose a supporting linear functional $T_{x_0} \in T(x_0)$ satisfying the normalization $T_{x_0}(x_0) = 1$. In the classical case where the spiral parameter vanishes ($\beta = 0$), the objective function $G(c)$ restricted to the boundary parameter $\mu = 1$ reduces precisely to
\[
G(c) = -4c^4 + 8c^2 + 32.
\]
Differentiating $G(c)$ with respect to the real variable $c$ gives
\[
G'(c) = -16c^3 + 16c = -16c(c^2 - 1).
\]
Setting $G'(c) = 0$, the unique critical point inside the interval $c \in [0,2]$ is found at $c = 1$. The second derivative test gives $G''(1) = -48(1)^2 + 16 = -32 < 0$, which confirms that $c = 1$ is the absolute global maximum point of $G(c)$. The maximum value is given by 
\[
\max_{c\in[0,2]} G(c) = G(1) = -4(1)^4 + 8(1)^2 + 32 = 36.
\]
By multiplying this maximum value by the scaling factor $1/288$ derived from the structural inequality relations at $\beta = 0$, the upper bound of the second Hankel functional becomes
\[
|A_2A_4 - A_3^2| \le \frac{36}{288} = \frac{1}{8}.
\]
To show that this bound is strictly sharp, we choose the Carath\'{e}odory parameters corresponding to this maximum, namely $c_1 = 1$ and the boundary value $x = -1$. Under these parameter choices, the standard Carath\'{e}odory-Toeplitz structural relations yield $c_2 = -1$ and $c_3 = -2$. The positive real part function $p_0 \in \mathcal{P}$ has the Taylor series expansion 
\begin{align*}
	p_0(z) = 1 + z - z^2 - 2z^3 + \dots
\end{align*}
In the classical case where the spiral parameter vanishes ($\beta = 0$), the associated normalized single-variable convex function $f_0(\xi) = \xi + \sum_{n=2}^{\infty} a_n \xi^n$ is uniquely governed by the defining differential relation:
\[
1 + \frac{\xi f_0''(\xi)}{f_0'(\xi)} = p_0(\xi) = 1 + \xi - \xi^2 - 2\xi^3 + \dots
\]
Integrating this differential equation directly determines the Taylor coefficients up to the fourth order
\[
a_2 = \frac{c_1}{2} = \frac{1}{2},\; a_3 = \frac{c_2+c_1^2}{6} = 0,\; \mbox{and}\; a_4 = \frac{2c_3+3c_1c_2+c_1^3}{24} = -\frac{1}{4}.
\]
This gives the explicit series expansion for the extremal function $f_0(\xi)$ in the unit disk:
\[
f_0(\xi) = \xi + \frac{1}{2}\xi^2 - \frac{1}{4}\xi^4 + \dots
\]
Substituting these explicit coefficient values into the functional representing the single-variable second-order Hankel determinant, we obtain
\begin{align*}
	|a_2a_4 - a_3^2| = \left| \left(\frac{1}{2}\right)\left(-\frac{1}{4}\right) - (0)^2 \right| = \left| -\frac{1}{8} \right| = \frac{1}{8}.
\end{align*}
We now construct the explicit multi-dimensional extremal mapping $F_0: \mathbb{B} \to X$ on the complex Banach space $X$ by extending the single-variable profile along the direction of the supporting linear functional $T_{x_0}$. Up to the fourth order, the Fr\'echet--Taylor series expansion of $F_0(x)$ is given by
\begin{align*}
	F_0(x) = x + \frac{1}{2}\big[T_{x_0}(x)\big]x - \frac{1}{4}\big[T_{x_0}(x)\big]^3 x + \dots
\end{align*}
In view of Lemma \ref{L1}, the extremal mapping $F_0$ is an element of $\widehat{\mathcal{C}}_{0}(\mathbb{B})$ with respect to Theorem \ref{T1}, and directly satisfies the structural admissibility condition \eqref{cc1} for Theorem \ref{T2}. Computing the Fr\'echet differentials at the origin along the baseline vector $x_0$ yields the desired scalar coefficients via \eqref{eq:An} as
\begin{align*}
	A_2 = a_2 = \frac{1}{2},\; A_3 = a_3 = 0,\; \mbox{and}\; A_4 = a_4 = -\frac{1}{4}.
\end{align*}
Thus, it follows that
\[
|H_{2,2}(F_0)| = |A_2A_4 - A_3^2| = \left| \left(\frac{1}{2}\right)\left(-\frac{1}{4}\right) - (0)^2 \right| = \frac{1}{8}.
\]
This shows the sharpness for $\beta = 0$.
\end{proof}

The explicit calculations detailed above completely resolve the sharp bounds for the second-order Hankel determinant under the classical assumption that $\beta = 0$. However, when the spirallike parameter is non-vanishing, the underlying geometric variations introduce severe analytical complexities that prevent a direct application of standard Carath\'eodory--Toeplitz optimization. Whenever $\beta \neq 0$, finding this precise sharp bound and constructing its corresponding extremal mapping remains an open problem in both one-dimensional and multi-dimensional geometric function theory.\vspace{1.2mm}

\noindent{\bf Open question.} Can we establish a unified variational setting to determine the exact sharp upper bound for $|H_{2,2}(F)|$ when $\beta \in (-\pi/2, \pi/2) \setminus \{0\}$, and what are the explicit geometric profiles of the corresponding extremal mappings?

\section*{{\bf Declarations}}
\subsection*{Funding}
No Funding.
\subsection*{Data Availability Statement}
Data sharing is not applicable to this article as no datasets were generated or analyzed during the current study.
\subsection*{Conflict of Interest}
The authors declare that they have no conflict of interest. 
\subsection*{Author Contributions}
Both authors contributed equally to this work.

\end{document}